\documentclass[a4paper,10pt]{amsart}
\usepackage[utf8]{inputenc}
\usepackage{lmodern}
\usepackage[T1]{fontenc}
\usepackage{microtype}
\usepackage[pdftex]{lscape}
\usepackage[all]{xy}
\usepackage{amsmath,amssymb,amsfonts,amsthm,latexsym,mathrsfs}
\usepackage{mathtools}
\usepackage{ stmaryrd } 
\usepackage{xcolor}
\usepackage{hyperref}
\usepackage{cleveref}
\hypersetup{colorlinks=true,linkcolor=gray,citecolor=gray}
\usepackage[hyperpageref]{backref} 
\newtheorem{lemma}{Lemma}[section]
\newtheorem{proposition}[lemma]{Proposition}
\newtheorem{theorem}[lemma]{Theorem}
\newtheorem{corollary}[lemma]{Corollary}

\newtheorem{maintheorem}{Theorem}

\theoremstyle{definition}

\newtheorem{definition}[lemma]{Definition}
\newtheorem{conjecture}[lemma]{Conjecture}
\newtheorem{question}[lemma]{Question}
\theoremstyle{remark}
\newtheorem{remark}[lemma]{Remark}

\newtheorem*{remark*}{Remark}
\newtheorem*{problem*}{Problem}

\newcommand\Hom{{\rm Hom}}

\newcommand\End{\rm End}

\def\Id{\operatorname{Id}}

\DeclareMathOperator{\sign}{sign}

\DeclareMathOperator{\ad}{ad}

\title{Codimension Growth of Lie algebras with a generalized action}
\author{Geoffrey Janssens}

\address{(Geoffrey Janssens) \newline Departement Wiskunde, Vrije Universiteit Brussel,
Pleinlaan $2$, 1050 Elsene, Belgium \newline E-mail address: {\tt geofjans@vub.ac.be}}

\begin{document}

\begin{abstract}
Let $L$ be a finite dimensional Lie $F$-algebra endowed with a generalized action by an associative algebra $H$. We investigate the exponential growth rate of the sequence of $H$-graded codimensions $c_n^H(L)$ of $L$ which is a measure for the number of non-polynomial $H$-identities of $L$. More precisely, we construct the first example of an $S$-graded Lie algebra having a non-integer, even irrational, exponential growth rate $\lim_{n\rightarrow \infty} \sqrt[n]{c_n^{S}(L)}$. Hereby $S$ is a semigroup and an exact value is given. On the other hand, returning to general $H$, if $L$ is semisimple and also semisimple for the $H$-action we prove the analog of Amitsur's conjecture (i.e. $\lim_{n\rightarrow \infty} \sqrt[n]{c_n^{H}(L)} \in \mathbb{Z}$). Moreover if $H=FS$ is a semigroup algebra the semisimplicity on $L$ can be dropped which is in strong contract to the associative setting.
\end{abstract}

\maketitle

\newcommand\blfootnote[1]{%
  \begingroup
  \renewcommand\thefootnote{}\footnote{#1}%
  \addtocounter{footnote}{-1}%
  \endgroup
}

\blfootnote{\textit{2010 Mathematics Subject Classification}. 17B01, 17B70, 20C30}
\blfootnote{\textit{Key words and phrases}. Polynomial identity, Lie algebra, semigroup grading, representation theory of symmetric group, Amitsur conjecture}
\blfootnote{The author is supported by Fonds Wetenschappelijk Onderzoek - Vlaanderen (FWO).}

\section{Introduction}

Let $L$ be a finite dimensional Lie algebra over some field $F$ of characteristic $0$. The codimension growth of associative and non-associative algebras was studied heavily in the last 40-50 years (see for instance \cite{GiaZaibook, BelKarRow}). It provides an important tool to study the 'size' of the T-ideal of polynomial identities of $L$ in asymptotic terms. In recent years, one has investigated this problem for several classes of rings with taking into account some extra structure, such as being graded by a group. To do so, one has to give appropriate definitions for the identities and notions considered.

In this article we will endow $L$ with a {\it generalized action} of an associative unital $F$-algebra $H$, i.e. a homomorphism $\varphi: H \rightarrow \End_F(L)$ such that for every $h\in H$ one has the compatibility rule
\begin{equation} 
 h . [l_1,l_2] = \sum\limits_{i=1}^k [h'_{i} l_1 , h''_{i} l_2]
\end{equation}
for some $h'_{i}, h''_{i} \in H$. We are interested in $H$-identities and the information contained by them. Let us recall briefly some definitions. Let  $X =  \{ x_1, x_2, \ldots  \}$ be a set of non-commutative variables and $L\left( X \mid H  \right)$ be the absolutely free $H$-module Lie algebra on $X$ (see \cite[Section 1.3.]{ASGordienko5} for precise definitions).

A polynomial $f(x_{i_1}^{h_1}, \ldots,x_{i_n}^{h_n}) \in L(X \mid H)$ is called an $H$-polynomial identity (for short, $H$-PI) if $f(l_{i_1}^{h_1}, \ldots,l_{i_n}^{h_n})=0$ for all $l_{i_j} \in  L$. Let $\Id^H(L)$ be the $T$-ideal of $H$-identities of $L$. Then one considers the relatively free $H$-module algebra $L( X \mid H) / \Id^H(L)$ and denotes by $c_n^H(L)$ the dimension of the subspace of multilinear elements in $n$ free generators. The sequence $c_n^H(L),~n=1,2,\ldots$ is called the sequence of {\it $H$-codimensions of $L$.} 

If $L$ satisfies an ordinary polynomial identity (i.e. $H = F$.) and $H$ is finite dimensional then it turns out that $c_n^H(L) \leq (\dim \varphi(H))^n c_n(L)$, where $ c_n(L) := c_n^{F}(L)$, holds for all $n$ (the proof is analog to \cite[Lemma 2]{ASGordienko5})). Since we assume $L$ to be finite dimensional,  $c_n(L)$ is exponentially bounded \cite[Theorem 12.3.11]{GiaZaibook} and hence by the above $c_n^H(L)$ also. Therefore it makes again sense to wonder what is the exponential growth rate of the sequence $(c_n^H(L))_n$. For associative algebras, without consideration of any extra action, Amitsur made in the 1980's a conjecture. It is natural to ask  the analogue in our setting hereof.

\begin{question}[Amitsur's conjecture] \label{amitsur voor generalized action vraag}
Let $L$ be a finite dimensional Lie algebra with a generalized action by an associative algebra $H$. When does $\lim\limits_{n \rightarrow \infty} \sqrt[n]{c_n^{H}(L)}$ exists and is moreover an integer?
\end{question}

If the limit exists we called the number the {\it $H$-exponent of $L$}. Above question was answered positively by Zaicev \cite{Zai, ZaiInt} in the classical setting. Furthermore he provided a precise formula connecting it thitly with the algebraic structure of $L$. Previously, it had been solved in case $L$ was solvable, semisimple or its solvable radical coincides with the nilpotent radical in respectively \cite{MisPet, GiaRegZai,GiaRegZai2}

Afterwards Gordienko generalized Zaicev's theorem by, amongst others, including gradations of finite abelian groups \cite{Gorliefirst} and, most recently, to $H$-module Lie algebras where $H$ is some finite dimensional semisimple Hopf algebra \cite[Theorem 10]{ASGordienko5}. As a particular case he proved that the graded exponent $\lim_{n \rightarrow \infty} \sqrt[n]{c_n^G(L)} = \exp^G(L)$ exists and is an integer for any group $G$. By a careful analysis of the existing proofs one can prove that this even hold if $L$ is graded by some cancellative semigroup. However in this paper we will not explain this further.

In the first part of the paper we aim at providing a counterexample to the variant of Amitsur's conjecture in the case of a grading by a semigroup, say $S$ (which alternatively may be viewed as a generalized action by $(FS)^{*}$, the dual of the semigroup algebra $FS$. Note that this is a bialgebra which is a Hopf algebra if and only if $S$ is in fact a group). More concretely we obtain the following result.

\begin{maintheorem}[\Cref{th: irrational lie alg example}]
Let $L$ be the $(\mathbb{Z}_2,\cdot)$-graded Lie algebra constructed in \Cref{frac expo sectie}. Then $\lim\limits_{n \rightarrow \infty} \sqrt[n]{c_n^{\mathbb{Z}_2}(L)} = 2 + 2\sqrt{2}$.
\end{maintheorem} 

It is important to remark that our counterexample is {\it not} graded semisimple. Interestingly, in \Cref{ss niet nodig in graded case} we point out that {\it if $L$ is  (semigroup) graded-semisimple, then \Cref{amitsur voor generalized action vraag} is true}. This is in strong contrast with the associative case. Indeed, even stronger, in \cite{ASGordienko13, GorJanJes} a class of semigroup-graded simple algebras was constructed with non-integer exponent. Moreover, approximately, any square root of an integer can be realised as the PI-exponent of such an algebra. We would also like to mention that along the way we prove a semigroup-graded version of Ado's theorem, see \Cref{graded ado theorem}.

In view of the deficiency just mentioned, we investigate in \Cref{sectie positive results Exponent Lie algebras} in general the case that $L$ is $H$-semisimple for a generalized action by an arbitrary associate algebra $H$. Hereby we obtain the following result:

\begin{maintheorem}[\Cref{exp is dim voor H-simple}, \Cref{H-ss integer exp}]
Let $L$ be an $H$-semisimple, semisimple Lie algebra and let $L= L_1 \oplus \ldots \oplus L_m$ with $L_i$ an $H$-simple algebra . Then 
$$\exp^H(L) = \lim\limits_{n \rightarrow \infty} \sqrt[n]{c_n^{H}(L)} = \max_{1 \leq i \leq m} \{ \dim_F L_i \}.$$
\end{maintheorem}

Hence it was exactly that lack of graded structure theory that enabled our counterexample.

As mentioned above, our counterexample shows that the $H$-semisimple condition is necessary for a positive answer on \Cref{amitsur voor generalized action vraag}. However, we do not know whether the condition that $L$ must be semisimple is necessary. In fact we expect it isn't. More precisely, we expect that the right condition is the one of $H$-nice, in the spirit of Gordienko \cite[section 1.7.]{ASGordienko5}. The condition of $H$-nice expresses that the classical structural results have an $H$-version (i.e. nilpotent and solvable radical are $H$-invariant, $H$-version of Levi decomposition, Wedderburn-Malcev and Weyl theorem). Therefore we formulate the following conjecture which is a variant of Regev and Amitsur's conjecture.

\begin{conjecture}
Let $F$ be an algebraicaly closed field with $\text{char}(F)=0$ and let $L$ be a finite dimensional Lie algebra with a generalized action by an associative algebra $H$. If $L$ is $H$-nice then there exist constants $C\in \mathbb{R}, t \in \frac{\mathbb{Z}}{2}, d \in \mathbb{Z}$ such that
$$c_n^{H}(L) \simeq C n^{t} d^n.$$
\end{conjecture}

In case that $H$ is a Hopf algebra, it was proven in \cite[Theorem 9]{ASGordienko5} that $\lim\limits_{n \rightarrow \infty} \sqrt[n]{c_n^{H}(L)}$ exists and is an integer. However, even in the classical setting of a trivial action, the statement about polynomial growth rate $t$ is still open! The main reason for this is the lack of knowledge about the existence of Kemer-type polynomials for Lie algebras. Any progress hereon would be very interesting. \vspace{0,3cm}

\noindent \textbf{Acknowledgment.} 
The author would like to warmly thank Alexey Gordienko for proposing the initial problem. He is also thankful to Mikhail Zaicev for expressing interest in preliminary results which further stimulated this research. \vspace{0,2cm}

\noindent {\it Conventions.}
Throughout the full paper we will assume (except stated explicitly otherwise) and denote the following:
\begin{itemize}
\item $F$ is a field of characteristic $0$,
\item  $L$  a finite dimensional Lie algebra over $F$,
\item  the commutator $[\cdot, \cdot]$ for multiplication in $L$,
\item  all commutators will be left-normed, i.e.\ $[x_1, \ldots, x_n]=[[x_1, \ldots, x_{n-1}],x_n]$
\item  $\langle a,b \rangle_L$ is the Lie algebra over $F$ generated by $a,b \in L$,
\item  $A^{[-]}$ denotes the Lie algebra associated to an associative algebra $A$.
\end{itemize}

\section{A graded non-integer Exponent} \label{frac expo sectie} \label{sectie non-integer Exp Lie algebras}
For the necessary background on graded polynomial identities and the graded free Lie algebra we refer the reader to \cite[Section 1.1.]{ASGordienko5} (all the definitions there are formulated for group-grading however they stay unchanged for semigroup gradings). Note that the infinite family of associative algebras $A$ constructed in \cite{GorJanJes} can not be used to construct a counterexample, because their $T$-gradation does not yield a gradation on $A^{[-]}$. \vspace{0,1cm}

{\it Construction of the graded Lie algebra.}\vspace{0,1cm}

\noindent We now define the main protagonist of this section. Let 
$$t := e_{12}= \left( \begin{array}{ll}
0 & 1 \\
0 & 0 \\
\end{array}\right), v := e_{21} = \left( \begin{array}{ll}
0 & 0 \\
1 & 0 \\
\end{array}\right), u :=e_{11} - e_{22}= \left( \begin{array}{ll}
1 & 0 \\
0 & -1 \\
\end{array}\right)$$
 and $I = \text{span}_F \{ u,v,t \}$. Note that $[u,v]=-2v,~[u,t]=2t$ and $[v,t]=-u$. Thus $I \cong \mathfrak{sl}_2(F)$, the Lie algebra consisting of the trace zero matrices over $F$. This is a simple Lie algebra of dimension 3. Denote by $\langle u,v \rangle_L $ the Lie subalgebra of $I$ generated by $\{ u,v \}$. Then we define $$L = I \oplus \langle u,v \rangle_L,$$
 with the usual bracket $[(a,b); (c,d)] = \left([a,c],[b,d] \right)$. Further we grade $L= L_0 \oplus L_1$ by the semigroup $(\mathbb{Z}_2,.)$ with $L_0 = (\mathfrak{sl}_2(F),0)$ and $L_1 = \{ (a,a) \mid a \in \langle u,v \rangle_L \}$. \vspace{0,1cm}

{\it Statement and Sketch proof.}\vspace{0,1cm}

\noindent We prove in this section that the graded exponent of $L$ is irrational. For ease of notation, we write in the remainder of the section $\exp^{\mathbb{Z}_2}(L)$ and $c_n^{\mathbb{Z}_2}(L)$ instead of respectively $\exp^{\mathbb{Z}_2\text{-}\mathrm{gr}}(L)$ and $c_n^{\mathbb{Z}_2\text{-}\mathrm{gr}}(L)$.
\begin{theorem}\label{th: irrational lie alg example}
Let $L$ be the Lie algebra with $(\mathbb{Z}_2,\cdot)$-grading as above. Then $\exp^{\mathbb{Z}_2}(L)= \lim\limits_{n \rightarrow \infty} \sqrt[n]{c_n^{\mathbb{Z}_2}(L)} = 2 + 2\sqrt{2}$.
\end{theorem} 

\begin{remark}\label{Het tegenvb is niet H-semisimple}
One verifies easily that $\mbox{Rad}(L)$, the solvable radical, equals $(0, \langle u ,v \rangle_L )$. Therefore, $L$ is not semisimple.  Moreover, since the only graded ideals of $L$ are $0, L$ and $(I,0)$, we see that $L$ also is not $(\mathbb{Z}_2,\cdot)$-semisimple (i.e.\ it is not the sum of graded-simple subalgebras).  Later on, in remark \ref{ss niet nodig in graded case}, we will note that if $L$ is graded-semisimple then the exponent is an integer. Finally, remark that $\mbox{Rad}(L)$ is not graded, which is an important difference with the group-graded case \cite[Prop. 3.3]{PagRepZai}. Actually it is this lack of graded structure theory that enables the current counterexample to the graded version of Amitsur's Conjecture.
\end{remark}

\noindent {\it Notation.} From now on, in order to avoid confusion with $(\mathbb{Z}_2,+)$-gradings, we denote $T = (\mathbb{Z}_2,\cdot)$.
 
\noindent The proof of Theorem \ref{th: irrational lie alg example} is a very natural one from an $S_n$-representation theory point of view. Namely consider $\frac{V_n^{S}(L)}{V_n^{S}(L) \cap \Id^{S}(L)}$
as $F S_n$-module, and decompose into a direct sum of Specht modules. Thus,
$$c_n^{T\text{-}\mathrm{gr}}(L) = \sum\limits_{\lambda \vdash n} m_{\lambda}^{T}(L) \dim_F S^F(\lambda),$$
where $m_{\lambda}^{T}(L)$ is the multiplicity of $S^F(\lambda)$. The proof now consists of the following three parts.

\textbf{(a)} First, we have to prove that the multiplicities $\sum_{\lambda \vdash n} m_{\lambda}^{T}(L)$ are bounded by a polynomial function.  This will be proven in Corollary \ref{bounded multiplicities} as a consequence of a graded version of Ado's Theorem \ref{graded ado theorem}. 

\textbf{(b)} Due to $(a)$, it is enough to estimate from above $\sum\limits_{\lambda \vdash n, m_\lambda \ne 0} \dim_F S^F(\lambda)$. Write $\lambda = (\lambda_1, \ldots, \lambda_l) \vdash n$ for a sufficiently large $n$. Then, using the Hook and Stirling formula one has that 
\begin{equation} \label{dimensie specht mod}
\begin{aligned}
\dim_F S^F(\lambda) & = \frac{n!}{\prod_{i,j} h_{\lambda}(i,j)} \leq \frac{n!}{\lambda_1! \cdots \lambda_l!}\\
& \simeq  \frac{\sqrt{2\pi}^{1-l} \sqrt{n} (\frac{n}{e})^n}{\sqrt{\lambda_1  \cdots  \lambda_l} (\frac{\lambda_1}{e})^{\lambda_1} \cdots (\frac{\lambda_q}{e})^{\lambda_l}}\\
 & =  \frac{\sqrt{2\pi}^{1-l} \sqrt{n}}{\sqrt{\lambda_1 \cdots  \lambda_l}} \left( \frac{1}{(\frac{\lambda_1}{n})^{\frac{\lambda_1}{n}} \cdots (\frac{\lambda_l}{n})^{\frac{\lambda_l}{n}}}\right)^n,
\end{aligned}
\end{equation}
for any partition $\lambda$ of $n$. Hence altogether we obtain that 
\begin{equation}\label{eq: de sup met de functie Lie example}
\limsup_{n\to \infty} \sqrt[n]{c_n^{T\text{-}\mathrm{gr}}(L))}
\leq \sup\limits_{\substack{\lambda \vdash n, \\ m_{\lambda}^{T}(L) \ne 0}} \Phi\left(\frac{\lambda_1}{n_1}, \ldots, \frac{\lambda_q}{n_q}\right).
\end{equation}
where $\Phi(x_1, \cdots, x_l) = \frac{1}{x_1^{x_1} \cdots  x_l^{x_l}}$ is a function on $\mathbb{R}^l$ that becomes continuous in the region $x_1, \cdots, x_l \geq 0$ if we define $0^0=1$. By restricting $\Phi$ to a region $\Omega$ having the property that ''if $\frac{\lambda}{n}=(\frac{\lambda_1}{n}, \ldots, \frac{\lambda_q}{n}) \notin \Omega $, then $m_{\lambda}^T(L) =0$'' we can lower the upper bound to $\limsup_{n\to \infty} \sqrt[n]{c_n^{T\text{-}\mathrm{gr}}(L))} \leq \max_{\vec{\alpha}\in\Omega} \Phi(\vec{\alpha})$. \Cref{region for the partitions} shows that if $ \lambda_6 > 0 \mbox{ and } \lambda_1 + 1 < \lambda_5 + \lambda_4$, then $m_{\lambda}^T(L) =0$. In particular we may take $$\Omega:=\left\{ (\alpha_1, \ldots, \alpha_5) \in \mathbb{R}^5 \mid \sum_{1 \leq i \leq 5} \alpha_i = 1,~ \alpha_1 \geq \ldots \geq \alpha_5 \geq 0,~ \alpha_{4} + \alpha_5 \leq \alpha_1 \right\}.$$
The value of $d$ is given in \Cref{waarde upper bound}.

\textbf{(c)} Since $c_n^{S}(L) \geqslant \dim M(\lambda)$ for all simple modules appearing in the decomposition, it is sufficient to find a partition $\mu = \mu_1 + \ldots + \mu_k$ such that $m(L, S, \mu) \neq 0$ and 
$$\dim_F S^F(\mu) \geqslant \frac{n!}{k^{k(k-1)}\mu_1! \ldots \mu_k!} \geqslant C n^{B} \left( \frac{1}{(\frac{\mu_1}{n})^{\frac{\mu_1}{n}}. \ldots .(\frac{\mu_k}{n})^{\frac{\mu_k}{n}}}\right)^n  \simeq C n^{B} d^n$$
for some constants $B, C \in \mathbb{R}$ in order to get the needed lower bound. In \Cref{alternating polynomial vb 1} we show that we can restrict $\Omega$ further to a region $\Omega_0$ such that $\max_{\Omega} \Phi = \max_{\Omega_0} \Phi$ and if $\frac{\lambda}{n} \in \Omega_0$ with $\lambda \vdash n$, then $m_{\lambda}^T(L) \neq 0$.
So, in this case, if $(\alpha_1, \ldots, \alpha_5)$ is an extremal point of $\Phi$ on $\Omega_0$, then the partition $\mu= (\mu_1, \ldots, \mu_k)$ with 
$$ \left\{ \begin{array}{ll}
\mu_i = \lfloor \alpha_{i}n \rfloor & \mbox{ for } 2 \leq i \leq k \\
\mu_1 = 1 - \sum_{i=2}^{k} \mu_i
\end{array}\right.$$
will have the right asymptotics and $m_{\mu}^T(L) \neq 0$, thus finishing the lower bound.

\subsection{Upper bound}\hspace{1cm}\vspace{0,1cm} \label{sectie Ado}

Recall that, by the Theorem of Ado, any finite dimensional Lie algebra has a finite dimensional faithful representation, i.e.\ there exists a Lie monomorphism $\rho: L \rightarrow \End_F(V)$ into the associated Lie algebra $\mathfrak{gl}_n(V) = \End_F(V)^{[-]}$, with $V$ a finite dimensional $F$-vector space. We prove that $A=\End_F(V)$ can be chosen such that a given gradation on $L$ 'is induced' from a gradation on $A$.

\begin{theorem}\label{graded ado theorem}
Let $L = \bigoplus_{t \in T} L^{(t)}$ be a finite dimensional Lie algebra graded by a finite abelian semigroup $T$. Then there exist a finite dimensional $T$-graded associative algebra $A= \bigoplus_{t \in T} A^{(t)}$ and a Lie monomorphism $\rho^{gr}: L \rightarrow A$ such that $\rho^{gr}(L^{(t)}) \subseteq A^{(t)}$ for all $t \in T$. 
\end{theorem}
\begin{proof}
As mentioned there exists a finite dimensional faithful Lie-representation $\rho: L \rightarrow \End_{F}(V)$. Further fix a vector space isomorphism  $\psi_t : V \rightarrow V^{(t)}$ for each $t \in T$ and define the finite dimensional $T$-graded vector space $V^{T} = \bigoplus_{t \in T} V^{(t)}$. Also denote $\End_{F}(V^{T})^{(t)} = \left\{ f \in \End_F(V^{T}) \mid f(V^{(s)}) \subseteq V^{(st)} \mbox{ for all } s \in T \right\}$ for all $t \in T$. The desired monomorphism is
$$\rho^{gr}:L \longrightarrow \bigoplus_{t \in T} \End_{F}(V^{T})^{(t)},$$
a map from $L$ to the outer direct sum $\bigoplus_{t \in T} \End_{F}(V^{T})^{(t)}$ that sends an arbitrary homogeneous element $l^{(t)} \in L^{(t)}$ to the linear map $\rho^{gr}(l^{(t)}): \bigoplus\limits_{s \in T} V^{(s)} \rightarrow \bigoplus\limits_{s \in T} V^{(s)}$ defined by the commutative diagram

\begin{displaymath}
\xymatrix{
V^{(s)} \ar[rr]^{\rho^{gr}(l^{(t)})} \ar[d]_{\psi_s}& & V^{(st)} \\
V \ar[rr]_{\rho(l^{(t)})}& & V \ar[u]_{(\psi_{st})^{-1}} \\
}
\end{displaymath}

One easily checks that $\rho^{gr}$ inherits from $\rho$ the faithfulness and property to be a Lie map. Clearly $\rho^{gr}$ satisfies the extra property $\rho^{gr}(L^{(t)}) \subseteq A^{(t)}$ where $A = \bigoplus\limits_{t \in T} A^{(t)} = \bigoplus\limits_{t \in T} \End_{F}(V^{T})^{(t)}$.\end{proof}

For a grading by the group $(\mathbb{Z},+)$ the theorem above was proven in \cite{Ros}. 

\begin{remark}
\begin{enumerate}
\item[(i)] In general $\End (W) \neq \bigoplus_{t \in T} \End(W)^{(t)}$ for a $T$-graded vector space $W$. This is the reason why we use the outer direct sum $\bigoplus_{t \in T} \End_{F}(W)^{(t)}$ in the proof of \Cref{graded ado theorem}. 
\item[(ii)] If $S$ is abelian, then the gradation of $A$ induces also a gradation on $A^{[-]}$. Moreover in this case $\rho$ is a graded lie morphism, i.e $\rho(L^{(t)}) \subseteq (A^{[-]})^{(t)}$ for all $t \in T$. 
\end{enumerate}
\end{remark}

As a direct consequence we get now that the multiplicities $\sum_{\lambda \vdash n} m_{\lambda}^T(L)$ are polynomially bounded.

\begin{corollary} \label{bounded multiplicities}
Let $L$ be a $T$-graded Lie algebra for some finite abelian semigroup $T$. Then there exist constants $C, d \in \mathbb{N}$ such that $\sum_{\lambda \vdash n} m_{\lambda}^T(L) \leq C n^d$ for all $n \in \mathbb{N}$.
\end{corollary}
\begin{proof}
By \Cref{graded ado theorem} there exists a finite dimensionsal associative algebra $A$ and a Lie monomorphism $\rho: L \rightarrow A^{[-]}$ such that $\rho(L^{(t)}) \subseteq A^{(t)}$ for $t \in T$. In particular 
$\sum_{\lambda \vdash n} m_{\lambda}^T(L) \leq \sum_{\lambda \vdash n} m_{\lambda}^T(A^{[-]})$, where $\frac{V_n^{T\text{-}\mathrm{gr}}(F)}{V_n^{T\text{-}\mathrm{gr}}(F)\cap \Id^{T\text{-}\mathrm{gr}}(A^{[-]})} = \bigoplus_{\lambda \vdash n} m_{\lambda}^{T}(A^{[-]}) S^F(\lambda)$.
Let $m_{\lambda}^T(A)$ be the multiplicity of $S^F(\lambda)$ in $\frac{P_n^{T\text{-}\mathrm{gr}}(F)}{P_n^{T\text{-}\mathrm{gr}}(F) \cap \Id^{T\text{-}\mathrm{gr}}(A)}$. Note that 
$$\sum_{\lambda \vdash n} m_{\lambda}^T(A^{[-]}) \leq \sum_{\lambda \vdash n} m_{\lambda}^T(A)$$
 since $V_n^{T\text{-}\mathrm{gr}}(F)$ is an $FS_{n}$-submodule of $P_{n}^{T\text{-}\mathrm{gr}}(F)$ and  $V_n^{T\text{-}\mathrm{gr}}(F) \cap \Id^{T\text{-}\mathrm{gr}}(A^{[-]}) = V_n^{T\text{-}\mathrm{gr}}(F) \cap \Id^{T\text{-}\mathrm{gr}}(A)$. Now, by \cite[Theorem 5]{Gor}, there exist constants $C, d \in \mathbb{N}$ such that $\sum\limits_{\lambda \vdash n} m_{\lambda}^T(A) \leq C n^d$. This finishes the proof. 
\end{proof}

\Cref{bounded multiplicities} can also be proven without making use of \Cref{graded ado theorem}. Actually one can rewrite word by word the proof of \cite[Theorem 12]{ASGordienko5} for $H= (FT)^*$. However this is lengthier.

Due to the strategy explained before, the upper bound will be a direct consequence of the following proposition.

\begin{proposition}\label{region for the partitions}
Let $L = \mathfrak{sl}_2(\mathbb{C}) \oplus \langle u,v \rangle_L$ be the $T=(\mathbb{Z}_2,.)$-graded Lie algebra defined at the beginning of \Cref{frac expo sectie}. Assume $m_{\lambda}^T(L) \neq 0$ for some partition $\lambda \vdash n$. Then $\lambda_6 = 0$ and $\lambda_1 + 1 \geq \lambda_5 + \lambda_4$.
\end{proposition}
\noindent {\bf Recurrent convention.}
Denote the $F$-basis of $L$ by $$\mathcal{B}_L = \{ (u,0),(u,u),(v,0),(v,v),(t,0) \}.$$ In the sequel we will always assume that the evaluations are from elements in $\mathcal{B}_L$.

\begin{proof}
Since $m_{\lambda}^T(L) \neq 0$ there exists a multilinear polynomial $f \in V_n^{T\text{-}\mathrm{gr}}(F)$ such that $e_\lambda f \notin \Id^{T\text{-}\mathrm{gr}}(L)$. 

\noindent Recall that $e^{*}_{\lambda} = \sum\limits_{\substack{\sigma \in R_{\lambda} \\ \tau \in C_{\lambda}}} \mbox{sgn}(\tau)~ \tau \circ \sigma$. Thus $e^{*}_\lambda f$ is alternating in the sets of variables corresponding to the numbers of each column of $T_{\lambda}$ and symmetric in those corresponding to the rows of $T_\lambda$. Thus, since $\dim_F L =5$ and $m_{\lambda}^T(L) \neq 0$, we must have that $\lambda_6 =0$ which we assume for the sequel of the proof.

Now define the function $\theta : L \rightarrow  \mathbb{Z}$ first on the basis elements by
$$\theta(u,u) = \theta(u,0)=0, \qquad \theta(v,v) = \theta(v,0) =1,~~ \text{ and } ~~ \theta(t,0)= -1.$$
and on an arbitrary element we take the maximum. Suppose $[b_1, \ldots, b_m] \neq 0$ for some basis elements $b_i \in \mathcal{B}_L$. One easily proves that 
$$-1 \leq \sum\limits_{1 \leq i \leq m} \theta(b_i) = \theta ([b_1, \ldots, b_m]) \leq 1.$$
Also $\sum\limits_{b \in \mathcal{B}_L} \theta(b) =1$ and $ \sum\limits_{b \in \mathcal{B}_L \setminus \{d \}} \theta (b) \geq 0$ for any $d \in \mathcal{B}_L$. Since $e^{*}_{\lambda} f \notin \Id^{T\text{-}\mathrm{gr}}(L)$ there exist some basis elements $b_1, \ldots, b_m \in \mathcal{B}_L$ such that $[b_1, \ldots, b_m] \neq 0$. By the previous inequalities we know that the $\lambda_4$ first columns of $T_\lambda$ give an altogether $\theta$-value of at least $\lambda_5$. Since the total $\theta$-value of $[b_1, \ldots, b_m]$ does not exceed $1$, there must remain at least $\lambda_5-1$ columns. Since the number of remaining columns is equal to $\lambda_1 - \lambda_4$ we get that $\lambda_1 - \lambda_4 \geq \lambda_5-1$ as desired.
\end{proof}
\begin{remark}
By interchanging the $\theta$-values of $u$ and $t$ one can prove analogously the above result for $L = \mathfrak{sl}_2(\mathbb{C}) \oplus \langle u,t \rangle_L $.
\end{remark}

As explained in the overview of the proof we have to compute the maximum of $\Phi(x_1, \ldots, x_q) = \frac{1}{x_1^{x_1} \ldots x_q^{x_q}}$ on the region
\begin{equation}\label{de regio bij Lie alg example}
\Omega =\left\{ (x_1, \ldots, x_q) \in \mathbb{R}^q \mid \sum_{1 \leq i \leq q} x_i = 1,~ x_1 \geq \ldots \geq x_q \geq 0,~ x_{q-1} + x_q \leq x_1 \right\}
\end{equation}
for $q=5$. This was already done in \cite[Lemma 3]{ASGordienko13}.

\begin{lemma} \label{waarde upper bound}
Let $q \in \mathbb{N}_{\geq 4}$. Then $\max_{\vec{x} \in \Omega} \Phi(\vec{x}) = (q-3) + 2 \sqrt{2} \approx q - 0.1716\ldots$
\end{lemma}

\begin{corollary} \label{upper bound frac vb 1}
$\limsup_{n\to \infty} \sqrt[n]{c_n^{T\text{-}\mathrm{gr}}(L)} \leq 2 + 2 \sqrt{2}$.
\end{corollary}

\subsection{Lower bound}\hspace{1cm}

Note that $\max_{\Omega} \Phi$, with $\Omega$ as in (\ref{de regio bij Lie alg example}), is reached at a point $(\alpha_1, \ldots, \alpha_5)$ with $\alpha_5 \neq 0$. Now we prove that $m_{\lambda}^T(L)\neq 0$ for all partitions $\lambda \vdash n$ with $\lambda_5 \neq 0$ and $\frac{\lambda}{n} \in \Omega$. So in this way we obtain the region $\Omega_0$ mentioned in the overview of the proof.

\begin{lemma} \label{alternating polynomial vb 1}
Suppose $\lambda_5 + \lambda_4 \leq \lambda_1$ and $\lambda_5 > 0$, then there exists a multilinear polynomial $f$ such that $e^{*}_{T_{\lambda}} f \notin \Id^{(FT)^{*}}(L)$ for a concrete Young tableau $T_{\lambda}$ constructed in the proof. 
\end{lemma}
\begin{proof}
Since $\lambda_5 + \lambda_4 \leq \lambda_1$ we can define numbers $\beta_2, \ldots, \beta_8 \in \mathbb{N}$ such that $\beta_2 = \lambda_4 - \lambda_5,\beta_3 + \beta_4 = \lambda_3 - \lambda_4, \beta_5 + \beta_6 = \lambda_2 - \lambda_3,$ $\beta_7 + \beta_8 = \lambda_1 - \lambda_2$ and $\beta_3 + \beta_5 + \beta_7 = \lambda_5$. We introduce these numbers to subdivide the columns of $T_\lambda$ in order to get more control of the different $\theta$-values of each column. Recall that, by the proof of Proposition \ref{region for the partitions}, we know that the total $\theta$-value has to be between $-1$ and $1$ for a non-zero valuation. Remark also that we need the condition $\lambda_5 + \lambda_4 \leq \lambda_1$ in order to be able to assume that all $\beta_i$ are greater than or equal to zero.

Now we define alternating multilinear $(FT)^{*}$-polynomials corresponding respectively to the $\lambda_5, \beta_2, \ldots, \beta_8$ first columns. Recall that $h_t$, $t \in T$, denotes the dual basis of $FT$, i.e.\ $h_t(s) = 1$ if $t=s$ and zero otherwise and $T = (\mathbb{Z}_2, \cdot)$.

 $$f_1 := \sum_{\sigma\in \text{Sym}\lbrace i_1, \ldots, i_5\rbrace} (\sign \sigma)
[x^{h_0}_{\sigma(i_2)}
,x^{h_0}_{\sigma(i_4)}
,x^{h_1}_{\sigma(i_3)}
,x^{h_0}_{\sigma(i_1)}
,x^{h_1}_{\sigma(i_5)}]
,$$ 
$$f_2 := \sum_{\sigma\in \text{Sym}\lbrace i_1, \ldots, i_4\rbrace} (\sign \sigma)
[x^{h_0}_{\sigma(i_2)}
,x^{h_0}_{\sigma(i_4)}
,x^{h_1}_{\sigma(i_3)}
,x^{h_0}_{\sigma(i_1)}]
,$$
$$f_3 := \sum_{\sigma\in \text{Sym}\lbrace i_1, i_2, i_3\rbrace} (\sign \sigma)
[x^{h_0}_{\sigma(i_2)}
,x^{h_0}_{\sigma(i_1)}
,x^{h_1}_{\sigma(i_3)}]
,$$ $$f_4 := \sum_{\sigma\in \text{Sym}\lbrace i_1, i_2, i_3\rbrace} (\sign \sigma)
[x^{h_0}_{\sigma(i_1)}
,x^{h_1}_{\sigma(i_3)}
,x^{h_1}_{\sigma(i_2)}]
,$$ $$f_5 := \sum_{\sigma\in \text{Sym}\lbrace i_1, i_2\rbrace} (\sign \sigma)
[x^{h_0}_{\sigma(i_1)}
,x^{h_0}_{\sigma(i_2)}]
,\qquad f_6 := \sum_{\sigma\in \text{Sym}\lbrace i_1, i_2\rbrace} (\sign \sigma)
[x^{h_0}_{\sigma(i_1)}
,x^{h_1}_{\sigma(i_2)}]
,$$ $$f_{7} := x^{h_0}_{i_1},\qquad f_{8} := x^{h_1}_{i_1}.$$
Finally, if $\beta_7 \neq 0$, define the polynomial
$$f=[(f_1f_3)^{\beta_3}, (f_1 f_5)^{\beta_5}, (f_1 f_7)^{\beta_7-1}, f_1, f_2^{\beta_2},f_4^{\beta_4},f_6^{\beta_6}, f_8^{\beta_8}, f_7] \in V_n^{(FT)^{*}}(F),$$
where by $[x,(ab)^{c}]$ we denote the polynomial $\underset{c-\mbox{times}}{[x,\underbrace{a,b},\ldots, a,b]}$.\newline If $\beta_7=0$ and $\beta_5 \neq 0$ then we define the polynomial 
$$f^{\prime}=[(f_1f_3)^{\beta_3}, (f_1 f_5)^{\beta_5-1}, f_1, f_2^{\beta_2},f_4^{\beta_4},f_6^{\beta_6}, f_8^{\beta_8}, f_5] \in V_n^{(FT)^{*}}(F) $$
and 
$$f^{\prime \prime}=[(f_1f_3)^{\beta_3-1}, f_1, f_2^{\beta_2},f_4^{\beta_4},f_6^{\beta_6}, f_8^{\beta_8}, f_3] \in V_n^{(FT)^{*}}(F) $$
if $\beta_5 = \beta_7 = 0$. Note that $\beta_3 \neq 0$ as $\lambda_5 = \beta_3 + \beta_5 + \beta_7 >0$.
Note that here different copies of $f_i$ depend on different variables. Thus: 

The copies of $f_1$ are alternating polynomials
of degree $5$ corresponding to the first $\lambda_5$ columns of height $4$.

The copies of $f_2$ are alternating polynomials of degree $4$ corresponding to
the next $\beta_2$ columns of height $5$.

\ldots

The copies of $f_{8}$ are polynomials of degree $1$ corresponding to the last
$\beta_{8}$ columns of height $1$. \newline However, the same values will be substituted.\newline\vspace{0.2cm}
Consider now the Young tableau $T_{\lambda}$ given by the figure below. We prove that $e_{T_{\lambda}}^{*} f$ does not vanish on $L$. First remark that $f \notin \Id^{(FT)^{*}}(L)$. Indeed the following substitution in $f$ is equal to a multiple of the element $(u,0)$. 

\begin{figure}[h]\caption{}\label{subst in poly vb 1}
$$T_{\lambda}=\begin{array}{|c|c|c|c|c|c|c|c|}
\multicolumn{1}{c}{\lambda_5} & \multicolumn{1}{c}{\beta_2} & \multicolumn{1}{c}{\beta_3} & \multicolumn{1}{c}{\beta_4} & \multicolumn{1}{c}{\beta_5} & \multicolumn{1}{c}{\beta_6} & \multicolumn{1}{c}{\beta_7} & \multicolumn{1}{c}{\beta_8}  \\
\hline
 (t,0) & (t,0) & (t,0) & (t,0) & (t,0) & (t,0) & (t,0) & (u,u) \\
 \cline{1-8}
(u,0) & (u,0) & (u,0) & (v,v) & (u,0) & (v,v) \\
 \cline{1-6}
 (u,u)& (u,u) & (u,u) & (u,u)  \\
 \cline{1-4}
 (v,0) & (v,0) \\
 \cline{1-2}
 (v,v) \\
 \cline{1-1}
\end{array}$$ 
\end{figure}

 (Here in the $i$-th block we have $\beta_i$ columns with the same values
in all cells of a row. For shortness, we depict each value for each block only once.
The tableau $T_{\lambda}$ is still of the shape $\lambda$.)

In fact one easily checks that after substitution $f_i$, for $i=3,5,7$, yields respectively $-8(t,0), 4(t,0)$ and $(t,0)$, for $i=2,4,6$, respectively $16(u,0), 2(u,0)$ and $2(u,0)$ and $f_1$ gives $-64(v,0)$. 

We claim that the substitution in $e^{*}_{T_{\lambda}} f$ as in figure (\ref{subst in poly vb 1}) is a non-zero multiple of the evaluated value of $f$. First remark, by construction of $f$, that $e^{*}_{T_{\lambda}} f = C a_{T_{\lambda}} f$ with $C =(5!)^{\lambda_5}(4!)^{\beta_2} (3!)^{\beta_3 + \beta_4} (2!)^{\beta_5 + \beta_6}$ and $a_{T_{\lambda}}$ symmetrizes $f$ corresponding to the rows of $T_{\lambda}$. Since $(t,0)$ and $(u,u)$ are in different homogeneous components all terms where $a_{T_{\lambda}} f$ interchanges a $(t,0)$ with $(u,u)$ will be zero. Similarly if a $(u,0)$ is interchanged with a $(v,v)$ in the second row, then this term is zero. So the claim and therefore the proposition are proven.
\end{proof}

\begin{corollary}\label{th: lie alg with non-integer Exponent}
With $L$ and $T$ as before, we have that
$$\exp^{T\text{-}\mathrm{gr}}(L):= \limsup_{n\to \infty} \sqrt[n]{c_n^{T\text{-}\mathrm{gr}}(L)} = 2 + 2 \sqrt{2}.$$
\end{corollary}
\begin{proof}
For the sake of completeness, we write how Lemma \ref{alternating polynomial vb 1} implies the lower bound, even though this was already  sketched before. Let $(\alpha_1, \ldots, \alpha_5) \in \mathbb{R}^5$ be an extremal point of the function $\Phi(x_1, \ldots, x_5) = \frac{1}{x_1^{x_1}. \ldots . x_5^{x_5}}$ on the polytope
$$
\Omega:=\left\{ (\alpha_1, \ldots, \alpha_5) \in \mathbb{R}^5 \mid \sum_{1 \leq i \leq 5} \alpha_i = 1,~ \alpha_1 \geq \ldots \geq \alpha_5 > 0,~ \alpha_{4} + \alpha_5 \leq \alpha_1 \right\}.
$$
By Lemma \ref{waarde upper bound}, $\Phi(\alpha_1, \ldots, \alpha_5) = 2 + 2\sqrt{2}.$ Define now the partition $\mu \vdash n$ by
$$ \left\{ \begin{array}{ll}
\mu_i = \lfloor n \alpha_{i} \rfloor & \mbox{ for } 2 \leq i \leq 5 \\
\mu_1 = 1 - \sum_{i=2}^{5} \mu_i.
\end{array}\right. $$

Since $(\alpha_1, \ldots, \alpha_5) \in \Omega$, the partition $\mu$ satisfies $\mu_4 + \mu_5 \leq \mu_1$ and $\mu_5 >0$. Thus, by Lemma \ref{alternating polynomial vb 1}, $m_{\mu}^T(L) \neq 0$. Moreover, for every $\epsilon > 0$ there exists a $n_0$ such that $\Phi(\frac{\mu_1}{n}, \ldots, \frac{\mu_5}{n}) \geq 2 + 2\sqrt{2} - \epsilon $ for all $n \geq n_0$. Now, for some constants $C_1, B_1 \in \mathbb{R}$
$$\dim_F (S^F_{\mu}) \geq \frac{n!}{n^{5.4}\mu_1! \ldots \mu_5!} \geq C_1 n^{B_1} \left( \frac{1}{(\frac{\mu_1}{n})^{\frac{\mu_1}{n}}. \cdots .(\frac{\mu_5}{n})^{\frac{\mu_5}{n}}}\right)^n  \geq C_1 n^{B_1} (d-\epsilon)^n,$$
which yields the lower bound $\liminf_{n\to \infty} \sqrt[n]{c_n^{T\text{-}\mathrm{gr}}(L)} \geq 2 + 2 \sqrt{2}$. Together with Corollary \ref{upper bound frac vb 1} we get that $\exp^{T\text{-}\mathrm{gr}}(L) =  \lim_{n\to \infty} \sqrt[n]{c_n^{T\text{-}\mathrm{gr}}(L)} = 2 + 2 \sqrt{2}$. 
\end{proof}

\section{Amitsur conjecture for $H$-semisimple Lie algebras}\label{sectie positive results Exponent Lie algebras}

In this section $H$ will always be a finite dimensional associative algebra with $1$ and $L$ a finite dimensional Lie algebra on which $H$ is acting in a generalized way, i.e.
\begin{equation} \label{gen action}
 h . [l_1,l_2] = \sum\limits_{i=1}^k [h'_{i} l_1 , h''_{i} l_2]
\end{equation}

 for some $h'_{i}, h''_{i} \in H$. We refer to \cite{ASGordienko5, Gor} for examples of generalized actions and for all basic definitions such as $H$-polynomials and $H$-codimensions. 
 
 \begin{definition}
The Lie algebra $L$ is called $H$-simple if it is non-abelian and the only $H$-invariant ideals of $L$ are $0$ and $L$. Furthermore $L$ is said to be $H$-semisimple if it is the direct sum of $H$-simple Lie algebras.
 \end{definition}
 If $L$ is $H$-simple, as $[L,L]$ is $H$-invariant, then $[L,L]=L$. As explained in \Cref{Het tegenvb is niet H-semisimple}, the Lie algebra $L$ with an irrational graded exponent constructed in \Cref{frac expo sectie} is not $H$-semisimple. The goal of this section is to contribute to the $H$-version of Amitsur's conjecture in the $H$-semisimple case without imposing restrictions on the acting algebra $H$. More concretely in \Cref{H-ss integer exp} we obtain a positive result if we moreover assume that $L$ is a semisimple Lie algebra.\vspace{0,1cm}
 
 {\it Non-polynomial Identities with enough alternations.}\vspace{0,1cm}
 
\noindent To start, recall that the adjoint representation, $\ad: L \rightarrow \End_F(L)$, of $L$ is defined as $\ad(l)(l^{\prime}) = [l,l^{\prime}]$ for all $l,l^{\prime} \in L$. We will sometimes write $\ad_l := \ad(l)$. Further, denote the map corresponding to the $H$-action by $\rho: H \rightarrow \End_F(L)$. Remark that by (\ref{gen action}) the following equality holds
\begin{equation}\label{actie H naar rechts bewegen}
\rho(h) \ad(l) = \sum_{i} \ad(h_i^{\prime} l) \rho(h_i^{\prime \prime}).
\end{equation}

\noindent Finally, by $Q_{t,k,n}^H \subseteq V_n^H$ we denote the subspace spanned by all multilinear $H$-polynomials alternating in $k$ disjoint sets $\{ x_1^i, \ldots, x_t^i \} \subseteq \{ x_1, \ldots, x_n \}$ of size $t$.

{\it Notational assumption:} In the sequel of this section we fix an $F$-basis $\mathcal{B}(L)= \{ l_1, \ldots, l_t \}$ of $L$.

First we prove that the necessary $H$-polynomial with sufficiently numerous alternations exists. The following is an analog of \cite[Theorem 1]{GiaSheZai}.

\begin{theorem}\label{non poly voor H-simple ss}
Let $L$ be a $H$-simple semisimple Lie algebra endowed with a generalized action of a finite dimensional associative algebra $H$ with $1$.  Then there exist a non-zero positive integer constant $C$ and $\overline{z}_1, \ldots, \overline{z}_C \in L$ such that for any $k$ there exists
$$f=f(x_1^1, \ldots, x_t^1; \ldots ; x_1^{2k}, \ldots, x_t^{2k}; z_1, \ldots, z_C ; z) \in Q^H_{t,2k,2kt+C+1}$$
such that for any $\overline{z} \in L$ we have $f(l_1, \ldots, l_t; \ldots ; l_1, \ldots, l_t; \overline{z_1}, \ldots, \overline{z_C}; \overline{z}) = \overline{z}$.
\end{theorem}

We will only give the proof in case $k=1$.  In order to obtain more alternating sets, i.e. $k > 1$, it became traditional to use a trick by Razmyslov \cite[Chapter III]{Razmyslov}. Despite that we are working with Lie algebra's with a generalized action, the proof is completely similar to those of Lemma 3 and Theorem 1 in \cite{GiaSheZai} where non-associative algebras without $H$-action are considered or also similar to the proof of \cite[Theorem 7]{ASGordienko3} where associative algebras endowed with a generalized action are considered.

\begin{proof}[Proof for $k=1$]
One can consider $L$ as module over its multiplication algebra $M(L) = \mbox{span}_F \{ \rho(H), \ad(L) \}$. Since $L$ is $H$-simple it is moreover an irreducible faithful module over $M(L)$ and so, as $L$ is finite dimensional, by the Density theorem $\End_F(L) = \mbox{span}_F \{ \rho(H), \ad(L) \}$.

Note that, due to (\ref{actie H naar rechts bewegen}), we always can move the $\rho(h)$ to the right in any expression in $\mbox{span}_F \{ \rho(H), \ad(L) \}$. Thus $\End_F(L) = \mbox{span}_F \{ \ad_l \circ \rho(h) \mid l \in L, h \in H \}$ and of course $\End_F(L) \cong M_t(F)$ as vector spaces, since $t = \dim L$. Let 
$$\mathcal{B}(\End_F(L)) = \{ \ad_{l_1}, \ldots, \ad_{l_t}; \ad(l_{i_1}) \rho(h_1), \ldots, \ad(l_{i_s}) \rho(h_s)\}$$  
be a basis of $\End_F(L)$ with $i_j \in \{ 1, \ldots, t \}$ appropriate indices.
Recall that by \cite{Formanek} the Regev polynomial
$$
\hspace{-0,17cm}\begin{array}{rl} f_t(x_1, \ldots, x_{t^2}; y_1, \ldots, y_{t^2}) = &\sum\limits_{\sigma, \tau \in S_{t^2}} \text{sgn} (\sigma \tau) x_{\sigma(1)} y_{\tau(1)} x_{\sigma(2)}x_{\sigma(3)}x_{\sigma(4)}y_{\tau(2)}y_{\tau(3)}y_{\tau(4)} \\ 
& \qquad \qquad \ldots x_{\sigma(t^2-2t + 2)} \ldots x_{\sigma(t^2)}y_{\tau(t^2-2t +2)} \ldots y_{\tau(t^2)} 
\end{array}$$
is a central polynomial of $M_t(F)$. Replace now each $x_1, \ldots, x_t$ by the respective $\ad_{x_i}$, $y_1,\ldots, y_t$ by $\ad_{y_i}$, $x_{t+j}$ by $\ad_{z_j} \circ \rho(h_j)$ and $y_{t+j}$ by $\ad_{v_{s+j}} \circ \rho(h_j)$ for $ 1 \leq j \leq s$, where $x_i, y_i, z_i$ are new variables that take values in $L$. Denote the polynomial that we get after this substitution by $\tilde{f}_t$. Note that if we evaluate $\tilde{f}_t$ by $x_i = y_i = l_i$ and $z_j = z_{s+j} = l_{i_j}$ then we get $K \mbox{id}_L$ for some non-zero constant $K \in F$. Finally put $C= 2s$, then clearly $f := K^{-1} \tilde{f}_t (z) \in V_n^H$ satisfies the needed properties. 
\end{proof} 

{\it Forumulas $H$-exponent.}\vspace{0,1cm}

\noindent In case of $H$-simple semisimple Lie algebras, the polynomial $f$ constructed in Theorem \ref{non poly voor H-simple ss} delivers now, in the classical way, that the $H$-exponent equals the dimension.

\begin{theorem}\label{exp is dim voor H-simple}
Let $L$ be an $H$-simple semisimple Lie algebra endowed with a generalized action of a finite dimensional unital associative algebra $H$. Then $\exp^H(L) = \dim L$.
\end{theorem}
\begin{proof}
Let $n \geq 2t+c+1$ and $k = \lfloor \frac{n - (2t+c+1)}{2t} \rfloor$. By Theorem \ref{non poly voor H-simple ss} there exists a $f \in Q^H_{t, 2k, n}$ which is a non-identity of $L$. We start by proving that there exists a $\lambda =(\lambda_1, \ldots, \lambda_h) \vdash n$ with $\lambda_i \geq 2k$ for $1 \leq i \leq h = t = \dim(L)$ such that $e_{\lambda} f \notin \Id^H(L)$ and in particular $m^H_{\lambda}(L) \neq 0$. 

It is well known that we can write $FS_n = \bigoplus\limits_{\substack{\lambda \vdash n, \\ T_{\lambda} \mbox{ {\tiny standard}}}} FS_ne^{*}_{T_{\lambda}}$. Consequently, as $f \notin \Id^H(L)$, there exists a $\lambda \vdash n$ such that $e^{*}_{T_\lambda} f \notin \Id^H(L)$. Moreover $\lambda_i \geq \lambda_{t} \geq 2k$. Indeed, $e^{*}_{\lambda} = b_{T_{\lambda}} a_{T_{\lambda}}$ and $a_{T_{\lambda}}$ is symmetrizing the variables of each row of $T_{\lambda}$, so each row of $T_{\lambda}$ may contain at most one variable from each $X_i = \{ x_1^{(i)}, \ldots, x_t^{(i)} \}$ since otherwise, $f$ being alternating in $X_i$, $a_{T_{\lambda}}f = 0$. Thus $\sum_{i=1}^{t-1} \lambda_i \leq 2k(t-1) + (n-2kt)= n-2k$ and $\lambda_t =  \sum_{i=1}^{t} \lambda_i - \sum_{i=1}^{t-1} \lambda_i = n - (n-2k)= 2k$ as claimed.

For this partition $c_n^H(L) \geq \dim_F S^F(\lambda)$. Also $((2k)^t) \leq \lambda$, i.e $D_{\lambda}$ contains the $t\times 2k$-box. By applying the Branching rule $n - 2kt$ times we see that $\dim_F S^F(\lambda) \geq \dim_F S^F((2k)^t))$. Finally 
$$\dim_F S^F((2k)^t)) \geq \frac{2kt!}{((2k+t)!)^t} \simeq \frac{\sqrt{4\pi tk}(\frac{2kt}{e})^{2kt}}{(\sqrt{2\pi(2k+t)}(\frac{2k+t}{e})^{2k+t})^t} \simeq C_1 k^{C_2} t^{2kt},$$
for some constants $C_1 \geq 0,~C_2 \in \mathbb{Q}$ as $k \rightarrow \infty$. This finishes the under bound.

The upper bound is also classical. For this consider $H$-polynomials as $n$-linear maps from $L$ to $L$. Then the map $V_n^H \rightarrow \Hom_F(L^{\otimes n}, L)$ has kernel $V_n^H \cap \Id^H(L)$. Thus $c_n^H(L) \leq \dim \Hom_F(L^{\otimes n}, L) = (\dim L)^{n+1}$.
\end{proof}

\begin{remark} \label{ss niet nodig in graded case}
\begin{enumerate}
\item Suppose $H = (FS)^*$ for some semigroup $S$. In this case, Theorem \ref{exp is dim voor H-simple} follows immediately from well known results and, moreover, one has not to assume $L$ to be semisimple as ungraded algebra. Indeed, in \cite[Proposition 1.12.]{ElKo} it is proven that if $L$ is $S$-graded-simple then $S$ is actually a commutative group. Moreover in \cite[Proposition 3.1]{PagRepZai} it is proven that $L$ is semisimple (as ungraded algebra) with isomorphic simple components whenever $L$ is group-graded-simple. Finally, by \cite[Theorem 1]{ASGordienko5} a finite dimensional Lie algebra graded by an arbitrary group satisfies the graded version of Amitsur conjecture and more precisely $\exp^G L = \dim_F L$ if $L$ is $G$-graded simple. 
\item It would be interesting to have an example of a generalized action by a bi-algebra which is not an Hopf algebra (which by the previous point is impossible for semigroup-algebras). If such an example would not exists, then by Gordienko's theorem \cite[Theorem 9]{ASGordienko5}, we can drop in \Cref{exp is dim voor H-simple} the condition that $L$ is semisimple.
\end{enumerate}
\end{remark}

\begin{corollary}\label{H-ss integer exp}
Let $L$ be an $H$-semisimple, semisimple Lie algebra and let $L= L_1 \oplus \ldots \oplus L_m$ with $L_i$ an $H$-simple algebra . Then $\exp^H(L) = \max_{1 \leq i \leq m} \{ \dim_F L_i \}$. Then $\exp^H(L) = \max_{1 \leq i \leq m} \{ \dim_F L_i \}$.
\end{corollary}
\begin{proof}
Since $L_i$ is a subalgebra of $L$, $\Id^H(L) \subseteq \Id^H(L_i)$ and thus 
$$\max\limits_{1 \leq i \leq t} \exp^H(L_i) \leq \liminf\limits_{n \rightarrow \infty} \sqrt[n]{c_n^H(L)}.$$

Now, note that since $L= \bigoplus_{i=1}^m L_i$ is a direct sum of Lie algebras, the Lie bracket can be seen as being the component-wise Lie bracket, i.e.\ $\left[ (l_1, \cdots, l_m),(l_1^{\prime}, \cdots, l_m^{\prime})\right] = \left([l_1,l_1^{\prime}], \cdots, [l_m, l_m^{\prime}]\right)$. Let $\mathcal{B}$ be a basis of $L$ consisting of the union of a fixed basis of each $L_i$. For a multilinear polynomial it is enough to evaluate basis elements in order to check whether it is a polynomial identity. Thus $V_n^H(F) \cap \Id^{H}(L) = V_n^H(F) \cap \bigcap_{i=1}^m \Id^H(L_i)$. 

By \cite[Theorem 12.2.13]{GiaZaibook} the statement now follows if $L$ satisfies $Q_{\max_i \dim L_i, k}= \bigcup_{n \in \mathbb{N}} Q_{\max_i \dim L_i, k,n}$, i.e.\ the Capelli identity of rank $\max_i \dim_F (L_i)+1$. However, by above remark this is clear.
\end{proof}

\bibliographystyle{plain}
\bibliography{bibLieAlgVb}

\end{document}